\documentclass[11pt,a4paper]{article}
\usepackage{fullpage}
\usepackage{amssymb,amsmath,stmaryrd,amsthm}
\usepackage[mathscr]{eucal}
\usepackage{pstricks,pst-node,pst-text,pst-3d}

\theoremstyle{definition}

\newtheorem{theorem}{Theorem}

\theoremstyle{remark}

\newcommand{\refe}[1] {(\ref{#1})}

\newcommand{\pp}[1] { \left( {#1} \right)  }

\newcommand{\pr}[1] { \left[ {#1} \right]  }

\def \dd {\,\mathrm{d}} 

\def\Ow{\mathrm{Ow}} 
\def\Bz{\mathrm{B}} 
\def\Lo{\mathrm{Lo}} 
\def \EE {\mathbb{E}} 
\def \Var {\mathrm{Var}} 
\def \RR {\mathbb{R}}

\def \0{{\mathbf{0}}}
\begin{document}

\title{A note on the Sobol' indices and interactive criteria}
\author{Michel GRABISCH${}^1$\thanks{Corresponding author.} and Christophe LABREUCHE${}^2$\\
\normalsize ${}^1$ Paris School of Economics, University of Paris I\\
\normalsize 106-112, Bd de l'H\^opital, 75013 Paris, France\\
\normalsize \tt michel.grabisch@univ-paris1.fr\\
\normalsize ${}^2$ Thales Research and Technology\\
\normalsize 1, Avenue Augustin Fresnel, 91767 Palaiseau, France\\
\normalsize \tt christophe.labreuche@thalesgroup.com}

\date{}
\maketitle

\begin{abstract}
The Choquet integral and the Owen extension (or multilinear extension) are the
most popular tools in multicriteria decision making to take into account the
interaction between criteria. It is known that the interaction transform and the
Banzhaf interaction transform arise as the average total variation of the
Choquet integral and multilinear extension respectively. We consider in this
note another approach to define interaction, by using the Sobol' indices which
are related to the analysis of variance of a multivariate model. We prove that
the Sobol' indices of the multilinear extension gives the square of the Fourier
transform, a well-known concept in computer sciences. We also relate the latter
to the Banzhaf interaction transform  and compute the Sobol' indices for the
2-additive Choquet integral.
\end{abstract}
{\bf Keywords:} capacity, interaction index, Sobol' index, multilinear extension,
  Fourier transform

\section{Introduction}
In multicriteria decision making, the Choquet integral with respect to a
capacity has become a popular tool to model situations where some interaction
exists between criteria \cite{grla07b}. This often happens in practice, as the
evaluation of an alternative under several criteria is a complex process, where
the level of importance of criteria generally depends on which criteria are
satisfied or not. 

The basic ingredient to model interaction is the capacity (or fuzzy measure)
defined on the set $N$ of criteria, through its interaction transform
\cite{gra96f}, which is a generalization of the Shapley value. Another type of
interaction has been introduced by Roubens \cite{rou96} under the name \textit{Banzhaf
  interaction transform}, and extends the Banzhaf value. So far, emphasis has
been put on the former one in theoretical developments and applications. 

It is well known that the Choquet integral is an extension of a capacity, seen
as a pseudo-Boolean function, and called the Lov\'asz extension. Another popular
extension of pseudo-Boolean functions is the Owen extension or multilinear
extension, known in Multi-Attribute Utility Theory from a long time ago
\cite{kera76}. Both can be considered as aggregation functions on a bounded
closed domain (say, $[0,1]^n$). Defining the interaction index for
$S\subseteq N$ w.r.t. an aggregation function $F$ on $[0,1]^n$ as the average of the
total variation of $F$ w.r.t. the coordinates in $S$, Grabisch et
al. \cite{grmaro99a} showed that the interaction transform corresponds to the
interaction index w.r.t. the Choquet integral, while the Banzhaf interaction
transform is the interaction index w.r.t. the multilinear extension. 

The present note aims to add another view of interaction to the picture, namely
the statistical view, and to relate it to a well-known transform used in
computer sciences but so far ignored in the field of decision making, which is
the Fourier transform. We show that considering the aggregation model as a
multivariate function and defining the interaction index of $S\subseteq N$
through the Sobol' index of $S$ (similar to the variance), we come up with the
square of the Fourier transform when the aggregation function is the multilinear
model (Theorem~1).  We also show the close relation between the Fourier transform and the
Banzhaf interaction transform,  and compute the Sobol' indices for the 2-additive
Choquet integral (Theorem~2).

Throughout the paper, cardinalities of sets $S,T,\ldots$ are denoted by the
corresponding small letters $s,t,\ldots$.

\section{Basic notions}
We consider throughout a finite set $N=\{1,\ldots, n\}$. We often denote
cardinality of sets $S,T,\ldots$ by corresponding small letters $s,t,\ldots$.

A \textit{set function} is a mapping $\xi:2^N\rightarrow\RR$. A \textit{game}
$v$ is set function vanishing on the empty set: $v(\varnothing)=0$. A
\textit{capacity} \cite{cho53} or \textit{fuzzy measure} \cite{sug74} $\mu$ is a
game satisfying monotonicity: for every $S,T\in 2^N$ such that $S\subseteq T$,
we have $\mu(S)\leqslant \mu(T)$.

Clearly, the set $\RR^{2^N}$ of set functions on $N$ is a $2^n$-dimensional
vector space. We introduce on $\RR^{2^N}$ the following inner product:
\[
\langle \xi,\xi'\rangle  = \frac{1}{2^n}\sum_{S\subseteq N}\xi(S)\xi'(S).
\]

A classical basis of $\RR^{2^N}$ is the basis of the
unanimity games. For any nonempty subset $S\subseteq N$, the \textit{unanimity
  game centered on $S$} is the game defined by
\[
u_S(T) = \begin{cases}
  1, & \text{if } T\supseteq S\\
  0, & \text{otherwise}
  \end{cases}.
\] 
Defining the set function $u_\varnothing(S) = 1$ for every $S\subseteq N$, we
obtain a basis of set functions. A drawback of this basis is that it is not orthogonal
w.r.t. the above inner product. We will introduce later an orthonormal basis.

It is well known that the coordinates of $\xi$ in this basis are the
  \textit{M\"obius transform} coefficients:
\[
\xi = \sum_{S\in 2^N}m^\xi(S) u_S
\]
with
\[
m^\xi(S) = \sum_{T\subseteq S}(-1)^{|S\setminus T|}\xi(T).
\]
The M\"obius transform (or \textit{M\"obius inverse}) is a fundamental notion in
combinatorics (see, e.g.,  \cite{rot64}).

There is another vision of set functions, namely the pseudo-Boolean functions
\cite{haru68}, noting that any subset $A$ of $N$ can be encoded by its
characteristic function $1_A$. Formally, a \textit{pseudo-Boolean function} is a
mapping $f:\{0,1\}^n\rightarrow \RR$.  If follows that the set of pseudo-Boolean
functions of $n$ variables is a $2^n$-dimensional vector space, with inner
product
\[
\langle f,f'\rangle  = \frac{1}{2^n}\sum_{x\in\{0,1\}^n}f(x)f'(x).
\]

The standard polynomial expression of a pseudo-Boolean function $f$ is
\begin{equation}\label{eq:2.pbf1}
f(x) = \sum_{A\subseteq N}f(1_A)\prod_{i\in A}x_i\prod_{i\in A^c}(1-x_i),
\end{equation}
for every $x\in\{0,1\}^n$, and with the convention
$\prod_{i\in\varnothing}x_i=1$. Rearranging terms, we get a sum of monomials:
\begin{equation}\label{eq:2.pbf2}
f(x) = \sum_{T\subseteq N}a_T\prod_{i\in T}x_i,
\end{equation}
for every $x\in\{0,1\}^n$, where the coefficients $a_T$ form the M\"obius
transform of $\xi_f$, the set function associated to $f$. Indeed, observe that
unanimity games $u_S$ correspond to monomials $\prod_{i\in S}x_i$.

Starting from (\ref{eq:2.pbf2}), the \textit{Owen extension} \cite{owe72} or
\textit{multilinear extension} is obtained by letting $x$ vary in $[0,1]^N$:
 \[
f^\Ow(x) = \sum_{T\subseteq N}a_T\prod_{i\in T}x_i \qquad (x\in [0,1]^N).
\]
Another possible extension is obtained by remarking that in (\ref{eq:2.pbf2}),
the product can be replaced by the minimum without changing the value of the
function. Letting again $x$ to vary in $[0,1]^N$, we obtain the \textit{Lov\'asz
  extension} \cite{lov83}:
\begin{equation}\label{eq:2.lovext}
f^\Lo(x) = \sum_{T\subseteq N}a_T\bigwedge_{i\in T}x_i \qquad (x\in[0,1]^N).
\end{equation}
The Lov\'asz extension coincides in fact with the Choquet integral
\cite{cho53}.

\section{Transforms of set functions}
The M\"obius transform presented above is an example of a linear and invertible
transform on the set of set functions, in the sense that to each set function
$\xi$ is associated another set function $m^\xi$, called the M\"obius transform
of $\xi$, being linear because $m^{\xi+\alpha\xi'} = m^\xi + \alpha m^{\xi'}$
for every $\xi,\xi'$ and $\alpha\in \RR$, and $\xi$ can be recovered from $m^\xi$
by the inverse transform:
\[
\xi(S) = \sum_{T\subseteq S}m^\xi(T) \qquad (S\in 2^N).
\]

\subsection{Interaction transforms}
Other such transforms exist, and three of them are of importance in this
paper. The first one is the \textit{interaction} transform \cite{gra96f}, defined by
\[
I^\xi(S)  =\sum_{T\subseteq N\setminus
S}\frac{(n-t-s)!t!}{(n-s+1)!}\sum_{L\subseteq S}(-1)^{|S\setminus L|}\xi(T\cup
L), 
\]
and the inverse relation is given by
\[
\xi(S) = \sum_{K\subseteq N}\beta_{|S\cap K|}^{|K|} I^\xi(K),
\]
where
\[
\beta_k^l = \sum_{j=0}^k\binom{k}{j}B_{l-j} \qquad (k\leq l),
\]
and $B_0,B_1,\ldots$ are the Bernoulli numbers.  Also, the interaction transform
has a simple expression in terms of the M\"obius transform:
\begin{equation}\label{eq:im}
I^\xi(S) = \sum_{T\supseteq S}\frac{1}{t-s+1}m^\xi(T).
\end{equation}

The interaction transform is of primary importance in multicriteria decision
making, as it permits to model interaction between criteria \cite{grla07b}.

A similar transform is the \textit{Banzhaf interaction transform} \cite{rou96} defined by
\[
I^\xi_{\mathrm{B}}(S) =  \Big(\frac{1}{2}\Big)^{n-s}\sum_{K\subseteq
  N}(-1)^{|S\setminus K|}\xi(K),
\]
with inverse relation
\begin{equation}\label{eq:invB}
(I_{\mathrm{B}}^{-1})^\xi(S) = \sum_{K\subseteq N}\Big(\frac{1}{2}\Big)^k(-1)^{|K\setminus S|}\xi(K).
\end{equation}
Its expression in terms of the M\"obius transform is
\begin{equation}\label{eq:2.banzhm}
I^\xi_{\mathrm{B}}(S) = \sum_{K\supseteq S}\Big(\frac{1}{2}\Big)^{|K\setminus S|}m^\xi(K),
\end{equation}
and the converse relation is
\begin{equation}\label{eq:2.mbanzh}
m^\xi(S) =\sum_{K\supseteq S}\Big(-\frac{1}{2}\Big)^{|K\setminus S|}I^\xi_{\mathrm{B}}(K).
\end{equation}

 In \cite{grmaro99a}, the following close relation between the two extensions of
a pseudo-Boolean function and the two interaction indices are shown:
\begin{align}
I^\xi_{\mathrm{B}}(S) &= \int_{[0,1]^n}\frac{\partial^s f^\Ow}{\partial x_S}(x)dx\label{eq:IBOw}\\
I^\xi(S) &= \int_{[0,1]^n}\Delta_S f^\Lo(x)dx\label{eq:ILo}
\end{align}
where $f$ is the pseudo-Boolean function corresponding to $\xi$ , $x_S$ is
the restriction of $x$ to coordinates in $S$, and 
\[
\Delta_Sf^\Lo(x) = \sum_{T\supseteq S}m^\xi(T)\bigwedge_{i\in T\setminus S}x_i,
\]
which plays the role of a partial derivative.

\subsection{Fourier transform}
Another transform is the \textit{Fourier transform}, well known in computer
sciences. It is defined as the coordinates of a set function in the basis of the
parity functions. For any subset $S\subseteq N$, the
\textit{parity function} associated to $S$ is the function
\begin{equation}\label{eq:2.par}
\chi_S(x) = (-1)^{1_S\cdot x} = (-1)^{\sum_{i\in S}x_i} \qquad (x\in\{0,1\}^n),
\end{equation}
where $1_S\cdot x$ is the inner product between the two vectors $1_S,x$. The
parity function outputs 1 if the number of variables in $S$ having value 1 is
even, and $-1$ if it is odd. In terms of set functions, the parity function
reads
\[
\chi_S(T) = (-1)^{|S\cap T|}\qquad (T\in 2^{N}).
\]
It can be checked that the set of parity functions forms an orthonormal basis of
$\RR^{2^N}$. 

Let us denote by $\widehat{f}(S)$, $S\subseteq N$, the
coordinates of a pseudo-Boolean function $f$ in the basis of parity functions:
\begin{equation}\label{eq:2.bfou}
f = \sum_{S\subseteq N}\widehat{f}(S)\chi_S.
\end{equation}
The basis being orthonormal, it follows from (\ref{eq:2.bfou}) that
$\widehat{f}(S)$ is simply given by
\begin{equation}\label{eq:2.four}
\widehat{f}(S) = \langle f,\chi_S\rangle
= \frac{1}{2^n}\sum_{x\in\{0,1\}^n}(-1)^{1_S\cdot x}f(x) \qquad (S\subseteq N),
\end{equation}
or, in terms of set functions,
\begin{equation}\label{eq:2.four1}
\widehat{\xi}(S) = \frac{1}{2^n}\sum_{T\subseteq N}(-1)^{|S\cap
T|}\xi(T) \qquad (S\subseteq N).
\end{equation}
The set of coordinates $\{\widehat{f}(S)\}_{S\subseteq N}$ is
the \textit{Fourier transform}.

We establish the relation between the Fourier, M\"obius and Banzhaf
transforms. Taking any set function $\xi$, we have
\begin{align*}
\widehat{\xi}(S) &= \frac{1}{2^n}\sum_{T\subseteq N}(-1)^{|S\cap T|}\xi(T)\\
& = \frac{1}{2^n}\sum_{T\subseteq N}(-1)^{|S\cap
  T|}\sum_{K\subseteq T}m^\xi(K) \\ 
& = \frac{1}{2^n}\sum_{K\subseteq
  N}m^\xi(K)\sum_{T\supseteq K}(-1)^{|S\cap T|}. 
\end{align*}
Now,
\begin{align*}
\sum_{T\supseteq K}(-1)^{|S\cap T|} &= (-1)^{|K\cap S|}2^{n-|K\cup S|} +
(-1)^{|K\cap S|+1}2^{n-|K\cup S|}\binom{|S\setminus K|}{1} + \\
 & \qquad (-1)^{|K\cap S|+2}2^{n-|K\cup S|}\binom{|S\setminus K|}{2} + \cdots +
(-1)^{|S|}2^{n-|K\cup S|} \\ 
 & = (-1)^{|K\cap S|}2^{n-|K\cup S|}\Big(1 - \binom{|S\setminus K|}{1} +
\binom{|S\setminus K|}{2} + \cdots + (-1)^{|S\setminus K|}\Big).
\end{align*}
Observe that 
\[
1 - \binom{|S\setminus K|}{1} +
\binom{|S\setminus K|}{2} + \cdots + (-1)^{|S\setminus K|}=0
\]
except if $|S\setminus K|=0$. It follows that
\begin{align}
\widehat{\xi}(S) &= \frac{1}{2^n}\sum_{K\supseteq S}m^\xi(K)(-1)^{|K\cap
  S|}2^{n-|K\cup S|}\nonumber\\
 & = (-1)^{|S|}\sum_{K\supseteq S}\frac{1}{2^k}m^\xi(K). \label{eq:2.foum}
\end{align}
Now, using (\ref{eq:2.banzhm}), we obtain
\begin{equation}\label{eq:2.fouB}
\widehat{\xi}(S) = \Big(\frac{-1}{2}\Big)^s I^\xi_\Bz(S).
\end{equation}
Lastly, we obtain from (\ref{eq:2.fouB}) and (\ref{eq:2.mbanzh})
\begin{equation}\label{eq:2.mfou}
m^\xi(S) = (-2)^s \sum_{T\supseteq S}\widehat{\xi}(T).
\end{equation}

\section{The Sobol' indices}
In statistics, the analysis of variance (ANOVA) (see, e.g., Fischer and
Mackenzie \cite{fima23}) is a well-known tool to model interaction between
variables in a multivariate model. Consider $n$ independent
random variables $Z_1,\ldots, Z_n$, with uniform distribution on $[0,1]$, and a
multivariate model $Y=F(Z)$, where $Z=(Z_1,\ldots,Z_n)$. Let us denote for
simplicity groups of variables $(Z_i)_{i\in S}$ by $Z_S$, and $Z_{-S}$ denotes
$(Z_i)_{i\not\in S}$. Hence, we may write $Z=(Z_S,Z_{-S})$. Moreover, we denote
by $\EE[Y]$ the expected value of $Y$ taken over all variables $Z_1,\ldots,
Z_n$. 
The expected value of $Y$ can be taken on a subset $Z_S$ of variables, with the corresponding notation $\EE_{Z_S}[Y]$.

Any multivariate function can be decomposed in the following way (ANOVA decomposition):
\[
Y=F(Z) = F_\varnothing+\sum_{i=1}^nF_i(Z_i) + \sum_{i<j}F_{ij}(Z_i,Z_j) + \cdots + F_N(Z) =
 \sum_{S\subseteq N}F_S(Z_S),
\]
with
\begin{align*}
F_\varnothing &=\EE[Y]\\
F_i(Z_i) & = \EE[Y|Z_i] - F_\varnothing\\
F_{ij}(Z_i,Z_j) & = \EE[Y|Z_i,Z_j] - F_i(Z_i) - F_j(Z_j) - F_\varnothing\\ &
= \EE[Y|Z_i,Z_j] - \EE[Y|Z_i]- \EE[Y|Z_j] + E[Y]\\
\vdots & = \vdots\\
F_S(Z_S) & = \EE_{Z_{-S}}[Y|Z_S] - \sum_{T\subset S}F_T(Z_T) = \sum_{T\subseteq
S}(-1)^{|S\setminus T|}\EE_{Z_{-T}}[Y|Z_T]\\
\vdots & = \vdots\\
F_N(Z) & = \sum_{T\subseteq N}(-1)^{|N\setminus T|}\EE_{Z_{-T}}[Y|Z_T].
\end{align*}
The property of this decomposition is that each term has zero mean, except the
first one, $F_\varnothing$. It follows that the variance of $Y$ can be decomposed
as follows:
\[
\Var[Y] = \sum_{\varnothing\neq S\subseteq N}\Var[F_S(Z_S)].
\]
The \textit{first-order Sobol' indices} \cite{sob90,sob93} are the quantities
$\frac{\Var[F_S(Z_S)]}{\Var[Y]}$, although one can omit the normalization
factor. The next theorem establishes the close link between Sobol' indices and
the Fourier transform (and consequently the Banzhaf transform) for the
multilinear model.
\begin{theorem}\label{th:6.sobol} 
Consider the multilinear extension $f^\Ow_\mu$ of a capacity $\mu$. Then the
(nonnormalized) Sobol' index for a subset $\varnothing\neq S\subseteq N$ is
given by
\[
\Var[(f^\Ow_\mu)_S] = \frac{1}{3^s}\big(\widehat{\mu}(S)\big)^2,
\]
where $\widehat{\mu}$ is the Fourier transform of $\mu$. 
Moreover, the ANOVA decomposition takes the following form
\[ f^\Ow_\mu(x) = \sum_{S\subseteq N} (-1)^s \: \prod_{i\in S} (2x_i-1) \times \widehat{\mu}(S) .
\] 
\end{theorem}
\begin{proof}
We set for simplicity $f=f^\Ow_\mu$. We compute
\[
f_S = \sum_{K\subseteq S}(-1)^k\EE(f\mid Z_{S\setminus K}) \qquad (S\subseteq N,
|S|>0).
\]
We have for any such $S$:
\begin{align}
\EE(f|Z_{S\setminus K}) &= \int_{[0,1]^{N\setminus (S\setminus K)}}f\dd
z_{N\setminus (S\setminus K)} = \sum_{T\subseteq
  N}m^\mu(T)\int_{[0,1]^{N\setminus (S\setminus K)}} \prod_{i\in T}z_i\dd
z_{N\setminus (S\setminus K)}\nonumber\\
 & =   \sum_{T\subseteq  N}m^\mu(T)\frac{1}{2^{|T\setminus (S\setminus K)|}}\prod_{i\in T\cap
  (S\setminus K)}z_i\nonumber\\
 & = \sum_{L\subseteq N\setminus S}\sum_{T\subseteq S}m^\mu(L\cup
T)\frac{1}{2^{|L\cup (T\cap K)|}}\prod_{i\in T\setminus K}z_i. \label{eq:6.eft}
\end{align}
It follows that
\begin{align*}
f_S &=\sum_{K\subseteq S}(-1)^k\sum_{L\subseteq N\setminus S}\sum_{T\subseteq S}m^\mu(L\cup
T)\frac{1}{2^{|L\cup (T\cap K)|}}\prod_{i\in T\setminus K}z_i\\
 & = \sum_{L\subseteq N\setminus S}\frac{1}{2^l}\sum_{T\subseteq S}m^\mu(L\cup
  T)\sum_{K\subseteq S}(-1)^k\frac{1}{2^{|T\cap K|}}\prod_{i\in T\setminus K}z_i.
\end{align*}
Letting $T'=T\cap K$, we have
\[
\sum_{K\subseteq S}(-1)^k\frac{1}{2^{|T\cap K|}}\prod_{i\in T\setminus K}z_i =
  \sum_{T'\subseteq T}(-1)^{t'}\frac{1}{2^{t'}}\prod_{i\in T\setminus
    T'}z_i\sum_{K'\subseteq S\setminus T}(-1)^{k'}.
\]
Observing that $\sum_{K'\subseteq S\setminus T}(-1)^{k'}=0$ except if
$S\setminus T=\emptyset$, it follows that
\begin{equation}\label{eq:6.fs}
f_S = \sum_{L\subseteq N\setminus S}\frac{1}{2^l}m^\mu(L\cup S)\sum_{T\subseteq
  S}(-1)^t\frac{1}{2^t}\prod_{i\in S\setminus T}z_i.
\end{equation}
Observe that 
\[
\sum_{T\subseteq S}(-1)^t\frac{1}{2^t}\prod_{i\in S\setminus T}z_i =
\frac{1}{2^{s}}\prod_{i\in S}(2z_i-1),
\]
hence we finally get by (\ref{eq:2.foum}):
\begin{equation}\label{eq:6.fsf}
f_S = (-1)^s\prod_{i\in S}(2z_i-1) \widehat{\mu}(S).
\end{equation}
Note that this expression is also true for $S=\emptyset$ as
\begin{align*}
 & f_\emptyset = \int_{[0,1]^N} f(z) \dd z = \sum_{T\subseteq N} m^\mu(T) \int_{[0,1]^N} \prod_{i\in T} z_i \dd z
  = \sum_{T\subseteq N} \frac{m^\mu(T)}{2^t} = \widehat{\mu}(\emptyset) .
\end{align*}  
We obtain finally
\[
\EE[f_S^2] = \int_{[0,1]^S} \Big(\prod_{i\in
S}(2z_i-1) \widehat{\mu}(S)\Big)^2\dd z_S =    \frac{1}{3^s}(\widehat{\mu}(S))^2.
\]
\end{proof}

This result is not very surprising as, historically, Sobol' generalized the Fourier base
to obtain the decomposition underlying the Sobol' indices \cite{sob93}.

Therefore, up to a multiplicative constant depending on the cardinality of the
subset, the Sobol' indices are the coefficients of the square of the Fourier
transform, or, due to (\ref{eq:2.fouB}), of the square of the Banzhaf transform
 (compare with (\ref{eq:IBOw})).

\medskip

The computation of the Sobol' indices for a general Choquet integral is quite complex due to the presence of the minimum operator in (\ref{eq:2.lovext}), and there does not seem to be a compact and appealing expression.
We restrict ourself to a sub-class of capacities -- called $2$ additive -- where all M\"obius coefficient of cardinality strictly greater than $2$ are zero \cite{gra96f}.
The Choquet integral of such capacities therefore becomes (see (\ref{eq:2.lovext}))
\[ f^\Lo(z) = \sum_{i\in N} m(i) \: z_i + \sum_{\{i,j\}\subseteq N} m(i,j) \: z_i \wedge z_j
\]
with the notation $m(i)=m(\{i\})$ and $m(i,j)=m(\{i,j\})$, the M\"obius
transform of the capacity. 

\begin{theorem}\label{th:sobol2} 
Consider the Lov\'asz extension $f^\Lo_\mu$ of a  $2$-additive capacity $\mu$
with M\"obius transform $m$.  
Then $f^\Lo_\mu$ is decomposed in the following terms according to ANOVA
\begin{align}
 &  f^\Lo_\emptyset = \sum_{i\in N} \frac{m(i)}{2} + \sum_{\{i,j\}\subseteq N} \frac{m(i,j)}{3}  \label{ELo1}  \\
 &   f^\Lo_k(z_k) = - \frac{m(k,\cdot)}{2} {z_k}^2 + \pp{m(k)+m(k,\cdot)} z_k - \pp{\frac{m(k)}{2}+\frac{m(k,\cdot)}{3}}  \label{ELo2} \\
 &  f^\Lo_{p,q}(z_p,z_q) = m(p,q) \pp{ -z_p\vee z_q + \frac{{z_p}^2}{2} + \frac{{z_q}^2}{2} + \frac{1}{3}} \label{ELo4}
\end{align}
where $m(k,\cdot) := \sum_{i\in N\setminus\{k\}} m(k,i)$.
Moreover, the (nonnormalized) Sobol' index are given by
\begin{align}
 & \Var[f^\Lo_k] = \frac{m(k)^2}{12} + \frac{m(k,\cdot)^2}{45} + \frac{m(k)m(k,\cdot)}{12}  \label{ELo3} \\
 & \Var[f^\Lo_{p,q}] = \frac{m(p,q)^2}{90}  \label{ELo5}
\end{align}
\end{theorem}
\begin{proof}
\[ f^\Lo_\emptyset = \sum_{i\in N} m(i) \: \EE[Z_i] + \sum_{\{i,j\}\subseteq N} m(i,j) \: \EE[Z_i \wedge Z_j]
\]
where
$\EE[Z_i \wedge Z_j] = 2 \int_0^1 \int_0^{z_i} z_i \wedge z_j \: dz_i \: dz_j
 = 2 \int_0^1 \int_0^{z_i} z_j \: dz_i \: dz_j = 2 \int_0^1 \frac{{z_i}^2}{2} \: dz_i = \frac{1}{3}$.
Hence \refe{ELo1} is proved.

\medskip

For $k\in N$,
\[ f^\Lo_k(z_k) = \sum_{i\in N} m(i) \: \pp{\EE[Z_i | Z_k=z_k]-\frac{1}{2}} + \sum_{\{i,j\}\subseteq N} m(i,j) \: \pp{\EE[Z_i \wedge Z_j | Z_k=z_k]-\frac{1}{3}}
\]
We observe that
$\EE[Z_k \wedge Z_i | Z_k=z_k] = \int_0^1 z_k \wedge z_i \: dz_i = \int_0^{z_k} z_i \: dz_i + \int_{z_k}^1 z_k \: dz_i
  = z_k - \frac{{z_k}^2}{2} $. 
Hence
\begin{align*}
  f^\Lo_k(z_k) & = m(k) \: \pp{z_k-\frac{1}{2}} + m(k,\cdot) \: \pp{z_k - \frac{{z_k}^2}{2}-\frac{1}{3}} \\
	 & = - \frac{m(k,\cdot)}{2} {z_k}^2 + \pp{m(k)+m(k,\cdot)} z_k - \pp{\frac{m(k)}{2}+\frac{m(k,\cdot)}{3}} .
\end{align*}
Hence \refe{ELo2} is proved.
Then
\begin{align*}
  \EE\pr{\pp{f^\Lo_k}^2} & = \int_0^1 \left[ m(k)^2 \: \pp{{z_k}^2-z_k+\frac{1}{4}} + m(k,\cdot)^2 \: \pp{{z_k}^2 + \frac{{z_k}^4}{4} + \frac{1}{9} - {z_k}^3 -\frac{2\, z_k}{3} + \frac{{z_k}^2}{3}} \right.   \\
 & \left. + 2 m(k)m(k,\cdot) \: \pp{{z_k}^2-\frac{{z_k}^3}{2} - \frac{z_k}{3}-\frac{z_k}{2}+\frac{{z_k}^2}{4}+\frac{1}{6}} \right] \: dz_k \\
& = \frac{m(k)^2}{12} + \frac{m(k,\cdot)^2}{45} + \frac{m(k)m(k,\cdot)}{12} 
\end{align*}  
Hence \refe{ELo3} holds.

\medskip

Consider now
\begin{align*} 
  f^\Lo_{p,q}(z_p,z_q) & = m(p,q) \pp{ z_p\wedge z_q - \EE[Z_p\wedge Z_q | Z_p=z_p]  - \EE[Z_p\wedge Z_q | Z_q=z_q] - \EE[Z_p\wedge Z_q ] }   \\
 & = m(p,q) \pp{ z_p\wedge z_q - \pp{z_p-\frac{{z_p}^2}{2}}  - \pp{z_q-\frac{{z_q}^2}{2}} + \frac{1}{3}} \\
 & = m(p,q) \pp{ -z_p\vee z_q + \frac{{z_p}^2}{2} + \frac{{z_q}^2}{2} + \frac{1}{3}}  
\end{align*}
as $z_p \wedge z_q + z_p\vee z_q = z_p + z_q$. Hence \refe{ELo4} is proved.
Then
\begin{align*} 
 & \EE\pr{\pp{f^\Lo_{p,q}}^2}  = m(p,q)^2 \int_0^1 \int_0^1 \pp{ -z_p\vee z_q + \frac{{z_p}^2}{2} + \frac{{z_q}^2}{2} + \frac{1}{3}}^2 \:dz_p \: dz_q   \\ 
 & = m(p,q)^2 \int_0^1 \int_0^1 \pp{ \pp{z_p\vee z_q}^2 - (z_p\vee z_q) \pp{{z_p}^2 + {z_q}^2 + \frac{2}{3}} }\:dz_p \: dz_q \\
 & + m(p,q)^2 \int_0^1 \int_0^1 \pp{  \frac{{z_p}^4}{4} + \frac{{z_q}^4}{4} + \frac{{z_p}^2\,{z_q}^2}{2} + \frac{1}{3}\pp{{z_p}^2+{z_q}^2} + \frac{1}{9}} \:dz_p \: dz_q  \\ 
 & = 2 \: m(p,q)^2 \int_0^1 \int_0^{z_p} \pp{ {z_p}^2 - z_p \pp{{z_p}^2 + {z_q}^2 + \frac{2}{3}} } \:dz_p \: dz_q \\
 & + m(p,q)^2 \pp{ \frac{1}{20} + \frac{1}{20} + \frac{1}{18} +\frac{1}{9} + \frac{1}{9} + \frac{1}{9}} \\
 & = 2 \: m(p,q)^2 \pp{ \frac{1}{4} - \frac{1}{5} - \frac{1}{15} - \frac{2}{9}}  + m(p,q)^2 \pp{ \frac{1}{10} + \frac{1}{18} + \frac{1}{3} } = \frac{m(p,q)^2}{90}
\end{align*}  
Hence \refe{ELo5} is proved.
\end{proof}

The interaction indices for the two-additive model are, using (\ref{eq:im})
\begin{align*}
 & I^\mu(k) = m(k) + \frac{m(k,\cdot)}{2} \\
 & I^{\mu}(p,q) = m(p,q)
\end{align*}
There is a clear difference between the Sobol' and the interaction indices, due to the presence of the square in the definition of the Sobol' indices.  
We observe that $I^{\mu}(p,q)$ and $\Var[f^\Lo_{p,q}]$ are both proportional to $m(p,q)$ or its square.
By contrast, $\Var[f^\Lo_k]$ is not proportional to the square of $I^\mu(k)$.
Term $m(k)$ takes more importance in $\Var[f^\Lo_k]$.

The Sobol' and interaction indices are both based on sensitivity analysis, but
performed in a different manner.  The interaction indices consider the average
value of the partial derivative of $f^\Lo$ w.r.t. its components
(see (\ref{eq:ILo})), and is thus the mean value of a local sensitivity analysis.
By contrast, the Sobol' indices comes from a sensitivity analysis based on
variance.  They are used for instance to identify which factors shall be fixed
(the other variables being unknown and governed by uniform distribution) in
order to reduce as much as possible the variance on the output variable.

It is not clear how to compare the interaction indices for singletons and pairs. For instance, if $I^\mu(1) = 2 \: I^{\mu}(1,2)$, does it mean that the interaction between variables $1$ and $2$ is twice more important than the importance of criterion $1$?
Such a comparison is possible with the Sobol' indices as the variance of the output variable is decomposed into the variance of each variable individually, each pair of variables, and so on. 
This allows to compare the Sobol' indices for different terms.
Then if $ \Var[f^\Lo_{1}] = 2\: \Var[f^\Lo_{1,2}]$, then one can say that variable $1$ alone is twice more influential than the interaction between variables $1$ and $2$.

%

\bibliographystyle{plain}
\bibliography{../BIB/fuzzy,../BIB/grabisch,../BIB/general}

\end{document}